\documentclass[reqno]{amsart}
\usepackage{amsthm, amsmath, amssymb,amscd}
\usepackage{xcolor,hyperref}
\usepackage{enumitem}
\theoremstyle{plain}
\newtheorem{theorem}{Theorem}[section]
\newtheorem{lemma}[theorem]{Lemma}

\newtheorem{corollary}[theorem]{Corollary}
\newtheorem{proposition}[theorem]{Proposition}
\theoremstyle{definition}
\newtheorem{definition}[theorem]{Definition}
\newtheorem{example}[theorem]{Example}

\theoremstyle{remark}
\newtheorem{remark}[theorem]{Remark}

\numberwithin{equation}{section}
\usepackage{todonotes}

\theoremstyle{definition}
\newtheorem{conjecture}[theorem]{Conjecture}

\allowdisplaybreaks

\begin{document}
	
	\title[On Monogeneity of reciprocal polynomials]
	{On Monogeneity of reciprocal polynomials}
	
		\author{Rupam Barman}
	\address{Department of Mathematics, Indian Institute of Technology Guwahati, Assam, India, PIN- 781039}
	\email{rupam@iitg.ac.in \\
		\href{https://orcid.org/0000-0002-4480-1788}{ORCID: 0000-0002-4480-1788}}
	
	\author{Anuj Narode}
	\address{Department of Mathematics, Indian Institute of Technology Guwahati, Assam, India, PIN- 781039}
	\email{anujanilrao@iitg.ac.in \\
		\href{https://orcid.org/0009-0005-6643-3401}{ORCID: 0009-0005-6643-3401}}
	
	\author{Vinay Wagh}
	\address{Department of Mathematics, Indian Institute of Technology Guwahati, Assam, India, PIN- 781039}
	\email{vinay\_wagh@yahoo.com \\
		\href{https://orcid.org/0000-0003-1977-464X}{ORCID: 0000-0003-1977-464X}}

	\date{\today}
	
	\thanks{}
	
	
	
	\dedicatory{}
	\begin{abstract}
		 Let $\mathbb{Z}_K$ denote the ring of integers of the number field $K = \mathbb{Q}(\theta)$, where $\theta$ is a root of the monic irreducible polynomial $f(x) \in \mathbb{Z}[x]$. We say that $f(x)$ is monogenic if $\mathbb{Z}_K = \mathbb{Z}[\theta]$. A polynomial $f(x) \in \mathbb{Z}[x]$ is called reciprocal if $f(x) = x^{\operatorname{deg}(f)} f(1/x)$. In this article, we derive sufficient conditions for the monogeneity of even degree reciprocal polynomials. By employing properties of the discriminant of reciprocal polynomials, we partially prove a conjecture proposed by Jones in $2021$. Furthermore, we establish a lower bound on the number of certain sextic monogenic reciprocal polynomials.
	\end{abstract}
	\maketitle
	\section{Introduction and Statements of Results}
	Let $f(x) \in \mathbb{Z}[x]$ be a monic irreducible polynomial with a root $\theta$ and $K = \mathbb{Q}(\theta)$. We denote by $\mathbb{Z}_K$ the ring of integers of $K$. It is well-known that $\mathbb{Z}_K$ is a free $\mathbb{Z}$-module of rank $\deg(f)$, and $\mathbb{Z}[\theta]$ is a submodule of $\mathbb{Z}_K$ of the same rank so that $[ \mathbb{Z}_K : \mathbb{Z}[\theta] ]$ is finite.
	The index of $f(x)$, denoted $\operatorname{ind}(f)$, is the index of $\mathbb{Z}[\theta]$ in $\mathbb{Z}_K$, i.e. $\operatorname{ind}(f)=[ \mathbb{Z}_K : \mathbb{Z}[\theta]]$.  
	The polynomial $f(x)$ is called monogenic if the index of $f(x)$ is one, i.e., $\mathbb{Z}_K = \mathbb{Z}[\theta]$. Further, the number field $K$ is said to be monogenic if there exists some $\alpha \in \mathbb{Z}_K$ such that $\mathbb{Z}_K = \mathbb{Z}[\alpha]$. Note that the monogeneity of the polynomial $f(x)$ implies the monogeneity of the number field $K$. However, the converse is false in general. 
	\par
	The discriminant  $\Delta(h)$ of a polynomial $h(x)$ is defined as 
	\begin{equation}\label{eq-1.1}
		\Delta(h) = (-1)^{n(n-1)/2} \hspace{3pt} a_n^{2n -2} \hspace{3pt} \prod_{i\neq j} (\alpha_i - \alpha_j),
	\end{equation} 
	where $ n = \deg(h),$ $a_n$ is the leading coefficient of $h(x)$, and the $\alpha_i$ are the roots of $h(x)$. 
	The discriminant $\Delta(K)$ of the number field $K$ and the discriminant of the polynomial $f(x)$ are related by the well-known equation
	\begin{equation} \label{eq-2.1}
		\Delta(f) = \operatorname{ind}(f)^2 \cdot \Delta(K).	
	\end{equation}
	Clearly, if $\Delta(f)$ is squarefree, then $f(x)$ is monogenic. However, the converse is false in general.
	\par 
	One of the classical and important questions in algebraic number theory is to determine whether a given number field is monogenic. The problem of characterizing such fields by explicit arithmetic conditions was first posed by Hasse \cite{hasse}. Since then, substantial progress has been made on various aspects of monogeneity. For more details, see the article of Ga\'al \cite{gaal-2} which offers a comprehensive overview of the recent progress in the study of monogeneity.
	\par
	In this article, we study monogeneity of reciprocal polynomials. A polynomial $f(x) \in \mathbb{Z}[x]$ is called reciprocal if $f(x) = x^{\operatorname{deg}(f)} f(1/x)$. For example, cyclotomic polynomials are reciprocal polynomials. Let $n \geq 3$ be an odd integer. Let $f(x)$ be a reciprocal polynomial of degree $n$. Then, $-1$ is a root of $f(x)$. Hence, $f(x) = (x+1) h(x)$, where $h(x)$ is a reciprocal polynomial of degree $n-1$. Therefore, to study reciprocal polynomials it is sufficient to consider even degree reciprocal polynomials. In \cite{Alexandersson}, Alexandersson \textit{et al.} studied the connection between even degree reciprocal polynomials and the Chebyshev polynomials of the first kind. The Chebyshev polynomials of the first kind are defined recursively as follows: 
	$$T_0(x) = 1, ~~ T_1(x) = x ~~ \text{and} ~~ T_n(x) = 2xT_{n-1}(x) - T_{n-2}(x) ~~ \text{for} ~~ n \geq 2.$$ 
We recall the following result of Alexandersson \textit{et al.} \cite{Alexandersson}.
		\begin{proposition} \cite[Proposition 3.2]{Alexandersson}\label{lemma_1}
		Let $n \geq 2$ be an integer and let 
		$$ f(x) = \sum_{j = 0}^{2n} a_j x^j \in \mathbb{Z}[x],$$ 
		where $a_j = a_{2n-j}$ for  all $j$, so that $f(x)$ is reciprocal. Define an nth degree polynomial $g(u) \in \mathbb{Z}[u]$, by 
		\begin{equation} \label{eq-2.2}
			g(u) := a_n + 2 \sum_{j =1}^{n} a_{n-j} T_j(u/2).
		\end{equation}
		Then, $f(x) = x^ng(x+ x^{-1})$.
	\end{proposition}
	We shall henceforth use the following notation. Let $f(x)$ be a reciprocal polynomial of degree $2n$, and let $g(x)$ be the degree $n$ polynomial obtained in Proposition~\ref{lemma_1} such that 
	$$
	f(x) = x^{n} g(x + x^{-1}).
	$$ 
	In \cite{jones_2021-2}, Jones constructed infinite families of monic sextic reciprocal monogenic polynomials $f(x)$, with $\text{Gal}(f)\cong D_n$, where $\text{Gal}(f)$ denotes the Galois group of $f(x)$ over $\mathbb{Q}$, and $D_n$ denotes the dihedral group of order $2n$ with $n\in \{3, 6\}$. In \cite{jones_2021-1}, he subsequently extended some of the results from \cite{jones_2021-2} to larger degree polynomials. He used cyclotomic polynomials to construct a class of monogenic reciprocal polynomials of degree $2^aq^b$, where $a$ and $b$ are positive integers, and $q \in \{ 3,5,7\}$. More precisely, he proved the following theorem.	
	\begin{theorem}\cite[Theorem 1.1]{jones_2021-1} \label{th-2.11}
		Let $a \geq 0$ and $b\geq 1$ be integers. Let $q \in \{ 3,5,7\}$ and let $N = 2^a q^b.$ Let $r \geq 3$ be a prime such that $r$ is a primitive root modulo $q^2$. Then there exist infinitely many primes $p$ such that
		$$\mathcal{F}_{N,p}(x) : = \Phi_N(x) + 4rq^2p x^{\phi(N)/2} $$
		is monogenic, where $\Phi_N(x)$ is the cyclotomic polynomial of index $N$. 
	\end{theorem}
Jones remarked that Theorem~\ref{th-2.11} can be extended to primes $q$ with $11\leq q\leq 211$, but this extension is conditional on the $abc$-conjecture for number fields. He also proposed a conjecture \cite[Conjecture 3.3]{jones_2021-1} concerning the reciprocal polynomial $\mathcal{F}_{N, p}(x)$. In this article, we establish a partial proof of that conjecture. To state our result, we first introduce some notation.
Let \begin{equation} \label{eq-2.5}
		\mathcal{F}_{N,t}(x) : = \Phi_N(x) + 4rq^2t x^{\phi(N)/2}, 
	\end{equation}
	where $N=2^aq^b$,  $q\geq 3$ is a prime, $a \in \{0,1\}$ and $ b\geq 1$ are integers, $r$ is some prescribed nonzero integer constant, and $t$ is an indeterminate. We use \eqref{eq-2.2} of Proposition~\ref{lemma_1} to construct the polynomials $g_a(u)$ corresponding to the reciprocal polynomials $\mathcal{F}_{2^aq,t}(x)$:
	\begin{equation}\label{eq-2.7}
		g_a(u) =
		\begin{cases}
			\begin{array}{c@{\quad}l}
				4rq^2t + 1 + 2 \sum_{j =1 }^{(q-1)/2} T_j(u/2) &\text{if} \hspace{5pt} a =0,\vspace{13pt} \\
				4rq^2t + (-1)^{(q-1)/2} \left(1 + 2 \sum_{j =1 }^{(q-1)/2} (-1)^j T_j(u/2) \right)  &\text{if}\hspace{5pt} a = 1.
			\end{array}
		\end{cases}
	\end{equation}
	In \cite[Conjecture 3.3]{jones_2021-1}, Jones formulated a conjecture concerning the discriminant of $\mathcal{F}_{2^aq,t}(x)$ and the irreducibility of certain polynomials associated with the discriminant of $g_a(u)$. In the theorem below, we confirm the conjectured expression for the discriminant of $\mathcal{F}_{2^aq,t}(x)$.
	\begin{theorem}\label{conjecture-2.12}
		Let $q \geq 3$ be a prime, and $N = 2^a q$ with $a\in\{ 0,1\}$. Suppose $\mathcal{F}_{2^aq,t}$ and $g_a(u)$ are as in \eqref{eq-2.5} and \eqref{eq-2.7} with $r$ some prescribed nonzero integer constant. Then
		\begin{equation} \label{eq-2.8}
			\Delta(\mathcal{F}_{2^aq,t}) = 
			\begin{cases}
				\begin{array}{c@{\quad}l}
					q(4qrt+1)(4q^2rt + (-1)^{(q-1)/2}) \Delta(g_0)^2 & \text{if} \hspace{5pt} a = 0, \vspace{13pt} \\ 
					q(4qrt+(-1)^{(q-1)/2}) (4q^2rt + 1)\Delta(g_1)^2 & \text{if } \hspace{3pt} a =1,
				\end{array}
			\end{cases}
		\end{equation}
		where $\Delta(g_a)$ is the discriminant of the polynomial $g_a(u)$.
	\end{theorem}
Next, we adapt the methods of Jones \cite{jones_2021-2, jones_2021-1} to obtain infinitely many monogenic polynomials of degrees $10$ and $5$ that were not considered in \cite{jones_2021-2, jones_2021-1}.
	\begin{theorem} \label{theorem-2.8}
		If the $abc$-conjecture for number field is true then there exist infinitely many primes $p$ such that
		$$ 
		f_p(x) := x^{10} + x^9 + x^8 + x^7 + (2p + 1)x^6 + (4p + 1)x^5 + (2p + 1)x^4 + x^3 + x^2 + x + 1
		$$
		and $$g_p(x)  := x^5 + x^4 -4x^3 -3x^2 + (2p +3)x + 4p + 1
		$$
		are monogenic. Moreover, if $p \equiv -1 \pmod 3$ then the Galois group of $g_p(x)$ is $S_5$.
	\end{theorem}
	We find that the set of primes not exceeding $100$ for which Theorem~\ref{theorem-2.8} holds is given by $\{ 3, 5,7,13,19,37,41,43,53,61,67,71,73,79,89,97\}$.
	\par 
The following theorem is a more general result about monogeneity of reciprocal polynomials. 
	\begin{theorem} \label{th-2.14}
		Let $n \geq 2$ be an integer, and let 
		$$f(x) = \sum_{j=0}^{2n} a_j x^j \in \mathbb{Z}[x]$$ 
		be a monic and irreducible polynomial. Let $a_j = a_{2n-j}$ for all $j$, so that $f(x)$ is reciprocal. Define an $n$th degree polynomial $g(x) \in \mathbb{Z}[x]$ by 
		$$g(x) := a_n + 2\sum_{j=1}^{n} a_{n-j} T_j\left(\frac{x}{2}\right).$$ 
		Then, $f(x)$ is monogenic if the following conditions hold:
		\begin{enumerate}
			\item[(1)] $f(-1)f(1)$ is squarefree,
			\item[(2)] $g(x)$ is monogenic.
		\end{enumerate}
	\end{theorem}
	In \cite{Jones-2022, Jones-2025}, Jones studied the monogeneity of a class of reciprocal polynomials obtained by composing $x^{2^{n-1}}$ with certain reciprocal polynomial. More recently, Kaur \textit{et al.} \cite{Kaur-2025}  studied the index of power  compositional polynomials. By combining their results with Theorem~\ref{th-2.14}, we obtain the following corollary for compositional reciprocal polynomials.
	\begin{corollary} \label{Prop-1.7}
		Let $f(x)$ be a monic reciprocal polynomial of degree $2n$ with integer coefficients. Let $k \geq 2 $ be an integer such that $f(x^k) $ is irreducible. Then $f(x^k)$ is monogenic  if 
		\begin{enumerate}
			\item $g(x)$ is monogenic,
			\item $f(-1)f(1) $ is squarefree, 
			\item $p$ does not divide the index of $f(x^k)$ for all primes $p \mid k$.
		\end{enumerate}
	\end{corollary}
	\begin{remark}
		 Theorem~\ref{th-2.14} and Corollary~\ref{Prop-1.7} highlight a significant reduction in complexity: the problem of establishing the monogeneity of the degree $2n$ polynomial $f(x)$ reduces to the simpler task of verifying the monogeneity of the degree $n$ polynomial $g(x)$.
	\end{remark}
The converse of Theorem~\ref{th-2.14} would assert that if $f(-1)f(1)$ is not squarefree, or if a prime 
$p$ divides the index of $g(x)$, then $f(x)$ is not monogenic. In the following, we prove that if a prime $p$ divides the index of $g(x)$, it also divides the index of $f(x)$; consequently, $f(x)$ is not monogenic.
\begin{proposition} \label{prop-3.1}
	Suppose $f(x)$ and $g(x)$ are as defined in Theorem~\ref{th-2.14}, and let $p$ be a prime. If $p$ divides the index of $g(x)$, then $p$ also divides the index of $f(x)$.
\end{proposition}
In \cite{Ravi, prabhakar}, Jakhar \textit{et al.} investigated the index of the polynomials
$$
f(x) = x^n + (bx+1)^m \text{ and } g(x) = x^{\,n-m}(x+b)^m + 1.
$$
Note that $f(x)$ and $g(x)$ are reciprocals of each other, i.e, $g(x) = x^n f(1/x)$, and therefore have the same discriminant. Thus, to determine the primes dividing the indices of $f(x)$ and $g(x)$, it suffices to study the index of either one. In the following, we show that a prime $p$ divides the index of $f(x)$ if and only if $p$ divides the index of $g(x)$.
\begin{proposition} \label{Prop-4.1}
	Suppose $f(x) = x^n + a_{n-1} x^{n-1} + \cdots + a_1 x + 1 \in \mathbb{Z}[x]$. Then, a prime $p$ divides the index of $f(x)$ if and only if $p$ divides the index of $x^n f(\frac{1}{x})$.
\end{proposition}
	In \cite{Anurabha-2025b}, Jakhar \textit{et al.} investigated the distinct squarefree parts of the discriminants of polynomials of the form 
	$$f(x) = x^n + c(at^k + b)^m,$$ where $a, b, c$ and $m$ are fixed integers and $n$ varies over $[1,N]$.
	As a consequence, they obtained lower bounds on the number of certain distinct quadratic fields.
	Furthermore, in \cite{Anurabha-2025a}, for the same family of polynomials, they established lower bounds on the number of monogenic polynomials in the cases (i) $m = 1$ and (ii) $b =1$ with $ m \geq 2$.
	Building on this line of work, we derive a lower bound for the number of monogenic reciprocal polynomials of degree six. More precisely, we prove the following result.
	\begin{theorem} \label{Prop-1.8}
		For an integer $a$, let $$f_a(x) = x^6 + x^5 + (2a + 1)x^4 + (4a + 1)x^3 + (2a+1)x^2 + x + 1.$$ Let $\mathcal{L}_{f}(N)$ denote the cardinality of the set 
		$$\{|a| \leq N \mid f_a(x) ~ \text{is monogenic}\}.$$ Then for sufficiently large $N$, we have $$\mathcal{L}_{f}(N) \gg \frac{N}{\log N}.$$
	\end{theorem}
	\section{Preliminaries}
	In this section, we state the $abc$-conjecture for number fields, which is a generalization of the classical $abc$-conjecture, and present some preliminary results that will be useful in proving our main results.
	Recall that a polynomial $f(x) \in \mathbb{Z}[x]$ is said to be reciprocal if $f(x) = x^{\operatorname{deg}(f)} f(\frac{1}{x}).$ Henceforth, we refer to $f(x)$ as a reciprocal polynomial of degree $2n$ and denote by $g(x)$  the polynomial constructed in  Proposition~\ref{lemma_1}, which satisfies the relation $f(x) = x^n g(x+x^{-1})$. One can easily observe the following properties of $f(x)$ and $g(x)$.
	\begin{enumerate}
		\item If $\alpha \in \mathbb{C}$ is a root of $f(x)$ then $1/\alpha $ is also a root of $f(x)$.
		\item If $\alpha$ is a root of $f(x)$ then $\alpha + {\alpha}^{-1}$ is a root of  $g(x)$.
		\item Roots $\alpha$ and $1/\alpha$ of $f(x)$ give the same root $\alpha + \alpha^{-1}$ of $g(x)$ with multiplicity two.
		\item If $f(x)$ is irreducible over $\mathbb{Z}$ then $g(x)$ is irreducible over $\mathbb{Z}$.
	\end{enumerate}
	 We now prove a lemma which establishes a relation between the discriminants of $f(x)$ and~$g(x)$.
	\begin{lemma} \label{lemma-2}
		Let $f(x)$ be a monic irreducible reciprocal polynomial of degree $2n$, and let $g(x)$ be the polynomial as defined in Proposition~\ref{lemma_1} such that $f(x)  = x^n g(x+x^{-1})$. Then,
		$$ \Delta(f) = (-1)^{n(2n-1)} f(1) f(-1) \Delta(g)^2.$$
	\end{lemma}
	\begin{proof}
		From \eqref{eq-1.1}, we have $$\Delta(f) = \prod_{i<j}^{2n} (\alpha_i - \alpha _j)^2 = (-1)^{2n(2n-1)/2}\prod_{i=1}^{2n}f'(\alpha_i),$$ 
		where the $\alpha_i$ are the roots of $f(x)$ in $\mathbb{C}$. We have 
		\begin{align*}
			\Delta(f) &= (-1)^{n(2n-1)}\prod_{i=1}^{2n}f'(\alpha_i) \\
			&= (-1)^{n(2n-1)} \prod_{i=1}^{2n} \left(\alpha_i^n g'(\alpha_i + \alpha_i^{-1})(1 - 1/{\alpha_i}^2) + n \alpha_i^{n-1}g(\alpha_i + \alpha_i^{-1})\right) \\ 
			&(\text{as } \alpha_i \hspace{2pt} \text{is a root of } f(x) \hspace{2pt}\text{we get} \hspace{2pt}  g(\alpha_i + \alpha_i^{-1} )=0)
			\\
			&= (-1)^{n(2n-1)}\prod_{i=1}^{2n} \alpha_i^n g'(\alpha_i + \alpha_i^{-1})(1 - 1/{\alpha_i}^2)
			\\
			&= (-1)^{n(2n-1)}\prod_{i=1}^{2n} \alpha_i^n (1 - 1/{\alpha_i}^2) \prod_{i=1}^{2n} g'(\alpha_i + \alpha_i^{-1}). 
		\end{align*}
		The roots $\alpha$ and $1/\alpha$ of $f(x)$ give the same root $\alpha + \alpha^{-1}$ of $g(x)$ with multiplicity 2. Further, as $f(x)$ is monic and reciprocal, therefore $\prod_{i=1}^{2n} \alpha_i =1$. Also note that, $\prod_{i=1}^{2n}(\alpha_i + 1) = f(1)$ and $\prod_{i=1}^{2n}(\alpha_i - 1) = f(-1)$.  This yields
		\begin{align*}
			\Delta(f)  &= (-1)^{n(2n-1)} ({\prod_{i=1}^{2n} \alpha_i})^{2n-2} (\prod_{i=1}^{2n} (\alpha_i + 1)(\alpha_i -1 ) ) (\prod_{i=1}^{n} g'(\alpha_i + \alpha_i^{-1}))^2 \\
			&=(-1)^{n(2n-1)} f(1)f(-1) \cdot \Delta(g)^2.
		\end{align*}
		This completes the proof of the lemma.
	\end{proof}
	In order to determine the Galois group in Theorem~\ref{theorem-2.8}, we make use of the  following two theorems. Theorem~\ref{Thm-2.8} is due to Dedekind.
	\begin{theorem}\cite[Theorem 1]{conrad} \label{Thm-2.8}
		Let $f(x) \in \mathbb{Z}[x]$ be a monic and irreducible polynomial of degree $n$. Let $p$ be a prime such that $p \nmid \Delta(f)$. If $f(x)$ factors in $\mathbb{F}_p$ as a product of distinct irreducible factors of degrees $n_1, n_2, \ldots, n_t$, then $\operatorname{Gal}(f)$, when viewed as isomorphic to a subgroup of the symmetric group $S_n$, contains a permutation $\alpha_1 \alpha_2 \cdots \alpha_t$, where $\alpha_i$ is a cycle of length $n_i$.
	\end{theorem}
	\begin{theorem}\cite[Theorem 2.2]{conrad-2} \label{Thm-2.9}
		For $n \geq 3$, a transitive subgroup of $S_n$ that contains a $3$-cycle and a $p$-cycle for some prime $p > n/2$ is $A_n$ or $S_n$.
	\end{theorem}
We recall the following proposition concerning the discriminants of the number fields $K\subset L$.
	\begin{proposition}\cite[Proposition 4.4.8]{cohen}{\label{theorem-4.1}}
		Let $K$ and $L$ be number fields with $K \subset L$. Then $$\Delta(K)^{[L:K]} | \Delta(L).$$
	\end{proposition}
	The following corollary is an immediate consequence of Proposition~\ref{lemma_1}. We require this in the proof of Theorem~\ref{theorem-2.8}.
	\begin{corollary}\cite[Corollary 2]{Elouafi-2014} \label{Coro-2.10}
		Every zero $\rho \neq \pm 2$ of $g(u)$ corresponds to the two distinct zeros $(\rho \pm \sqrt{\rho^2 - 4})/2$ of $f(x)$. In particular, every real zero of $g(u)$ in the interval $(-2,2)$ corresponds to a conjugate pair of distinct nonreal zeros of $f(x)$ on the unit circle. Moreover, if $g(u)$ and $f(x)$ are irreducible, it follows that $\operatorname{Gal}(g)$ is isomorphic to a subgroup of $\operatorname{Gal}(f)$.
	\end{corollary}
	We now state the $abc$-conjecture for number fields which is a generalization of the classical $abc$-conjecture. More details on the $abc$-conjecture for number fields can be found in \cite[Section 4]{pasten}.
	\begin{conjecture}[$abc$-conjecture for number fields] \label{conj-1}
		Let $K$ be a number field. Let $\epsilon > 0$ and fix mutually distinct elements $b_1, \ldots, b_m \in K$. Let $S$ be a finite set of places of $K$. Then, for all but finitely many $\alpha \in K$, one has 
		$$ (M- 2 \epsilon )h_K(\alpha)< \sum_{i=1}^{M}N_{K,S}^{(1)} ( \alpha - b_i ),$$
		where $h_K(\alpha) $ is the height (relative to K) of $\alpha \in K$ and  $N_{K,S}^{(1)}$ is the truncated counting function.
	\end{conjecture}
In Theorem~\ref{theorem-2.8}, the existence of infinitely many primes $p$ is ensured by the following result, derived from the work of Helfgott and Pasten. In particular, if we have a polynomial that factorizes as a product of distinct irreducible factors, we are interested in determining when there exist infinitely many primes $p$ such that $f(p)$ is squarefree. This is established by the following theorem and its corollary, as stated in \cite[Theorem 2.9 and Corollary 2.10]{jones_2021-2}. Before stating the results, we recall a relevant definition.
\begin{definition}
	 A polynomial $f(x)\in \mathbb{Z}[x]$ is said to have no local obstruction at a prime $p$ if there exists $l \in (\mathbb{Z}/p^2 \mathbb{Z})^*$ for which $f(l)$ is not divisible by $p^2$.
\end{definition}
	\begin{theorem} \label{thm-2.2}
		Let $f(x) \in \mathbb{Z}[x]$ be such that $f(x)$ is a product of distinct irreducible factors, where the largest degree of any irreducible factor is $d$. Define $$ N_f(X) = \# \{r \leq X  \mid r ~~ \text{is a prime and} ~~ f(r) ~~ \text{squarefree} \}.$$
		Then, the following asymptotic holds unconditionally if $d \leq 3$ and holds under the assumption of Conjecture~\ref{conj-1} for $f(x)$, if $d \geq 4$: 
		$$ N_f(X) \sim c_f  \frac{X}{\log X},$$
		where $$ c_f = \prod_{r ~~\text{prime}} \left( 1 - \frac{\rho_f(r^2)}{r(r-1)}\right)$$ and $\rho_f(r^2)$ is the number of $z \in (\mathbb{Z}/r^2 \mathbb{Z})^*$ such that $f(z) \equiv 0\pmod {r^2}$. The constant $c_f$ is positive if and only if $f(x)$ has no local obstruction at all primes $p$.
	\end{theorem}
	\begin{corollary} \label{coro-2.2}
		Let $f(x) \in \mathbb{Z}[x]$ be such that $f(x)$ is a product of distinct irreducible factors, where the largest degree of any irreducible factor is $d$. Further suppose that, for each prime $r$, there exists some $ z \in (\mathbb{Z} /r^2 \mathbb{Z})^{*}$ such that $f(z) \not \equiv 0 \pmod {r^2}.$ If  $d \leq 3$, or if $d \geq 4$ with the assumption of the $abc$-conjecture for number fields for $f(x)$, then there exist infinitely many primes $p$ such that $f(p)$ is squarefree.
	\end{corollary}
	
\section{Proof of Theorems \ref{conjecture-2.12}, \ref{theorem-2.8} and \ref{th-2.14}}
In this section, we prove Theorems~\ref{conjecture-2.12}, \ref{theorem-2.8} and \ref{th-2.14}. We first give a proof of Theorem~\ref{conjecture-2.12}.
		\begin{proof}[Proof of Theorem~\ref{conjecture-2.12}]
		We first consider the case $a=0$. Taking $a =0$ in \eqref{eq-2.5} yields 
		\begin{align*} 
			\mathcal{F}_{q,t}(1) = \Phi_{q}(1) + 4rq^2t = q ( 4rqt+1)
		\end{align*} 
		and 
		\begin{align*} 
			\mathcal{F}_{q,t}(-1) &= \Phi_{q}(-1) + 4rq^2t(-1)^{(q-1)/2} \\
			&= (-1)^{(q-1)/2}(4q^2rt + (-1)^{(q-1)/2}). \nonumber
		\end{align*}	
	Using Lemma~\ref{lemma-2} for the polynomial $\mathcal{F}_{q,t}$ we find that
		\begin{align*}
			\Delta(\mathcal{F}_{q,t}) &= (-1)^{{(q-1)(q-2)}/{2}} \mathcal{F}_{q,t}(1)\mathcal{F}_{q,t}(-1) \Delta(g_0)^2  \nonumber \\
			&=(-1)^{{(q-1)(q-2)}/{2}}q ( 4qrt+1) (-1)^{(q-1)/2} (4q^2rt + (-1)^{(q-1)/2}) \Delta(g_0)^2 \nonumber \\
			&= (-1)^{(q^2 -1)/2} q ( 4qrt+1)(4q^2rt + (-1)^{(q-1)/2}) \Delta(g_0)^2 \nonumber \\
			&=q ( 4qrt+1)(4q^2rt + (-1)^{(q-1)/2}) \Delta(g_0)^2. 
		\end{align*}
		This completes the proof of the theorem when $a=0$. 
		\par Next, we consider the case $a=1$. We note that if $q$ is an odd prime then
		$$\Phi_{2q}(x) = 1 -x +x^2 -x^3 + \cdots + x^{q-1} = \sum_{k =0 }^{q-1}(-x)^k.$$
		Taking $a=1$ in \eqref{eq-2.5}, we find that 
		\begin{equation*}
			\mathcal{F}_{2q,t}(1) = \Phi_{2q}(1) + 4rq^2t = 1 + 4q^2rt 
		\end{equation*}
		and 
		\begin{align*}
			\mathcal{F}_{2q,t}(-1) &=  \Phi_{2q}(-1) + 4rq^2t(-1)^{(q-1)/2} \\
			&=  (-1)^{(q-1)/2} q (4qrt + (-1)^{(q-1)/2}). \nonumber
		\end{align*}
	Employing Lemma~\ref{lemma-2} for the polynomial $\mathcal{F}_{q,t}$ we find that 
		\begin{align*}
			\Delta(\mathcal{F}_{2q,t}) &= (-1)^{{(q-1)(q-2)}/{2}} \mathcal{F}_{2q,t}(1)\mathcal{F}_{2q,t}(-1) \Delta(g_1)^2\\
			&= (-1)^{{(q-1)(q-2)}/{2}} ( 4q^2rt+1 ) (-1)^{(q-1)/2} q (4qrt + (-1)^{(q-1)/2}) \Delta(g_1)^2\\
			&= (-1)^{(q^2 -1)/2} ( 4q^2rt +1 ) q (4qrt + (-1)^{(q-1)/2}) \Delta(g_1)^2 \\ 
			&= q( 4q^2rt +1)  (4qrt + (-1)^{(q-1)/2}) \Delta(g_1)^2.
		\end{align*}
	This completes the proof of the theorem when $a=1$. 
	\end{proof}
Next, we prove Theorem~\ref{theorem-2.8}.
	\begin{proof}[Proof of Theorem~\ref{theorem-2.8}]
		For an integer $A$, we consider the following two polynomials
		\begin{equation*}
			f_A(x) = x^{10} + x^9 + x^8 + x^7 + (2A + 1)x^6 + (4A + 1)x^5 + (2A + 1)x^4 + x^3 + x^2 + x + 1
			\end{equation*} 
			and 
			\begin{equation*}
			g_A(x)   = x^5 + x^4 -4x^3 -3x^2 + (2A +3)x + 4A + 1.
		\end{equation*}
		We first show that there exist infinitely many primes $p$ for which $f_p(x)$ is monogenic.
		It is straightforward to verify that both $f_A(x)$ and $g_A(x)$ are irreducible modulo $2$, and therefore irreducible over $\mathbb{Q}$. Using \texttt{SageMath}, we compute
		\begin{equation}\label{eq_4.1}
		| \Delta(f_A) | = (8A + 11)  (8192A^5 + 125008A^4 + 156112A^3 - 139876A^2 - 15972A + 14641)^2,	
		\end{equation}
		\begin{equation}\label{eq_4.2}
		|\Delta(g_A) | =   8192A^5 + 125008A^4 + 156112A^3 - 139876A^2 - 15972A + 14641.
		\end{equation} 
		Observe that $|\Delta(f_A)| = | (8A+11) | \cdot | \Delta(g_A)|^2$. We next define the following polynomial
		$$
		h(y) := (8y + 11)(8192y^5 + 125008y^4 + 156112y^3 - 139876y^2 - 15972y + 14641).
		$$
		Note that $h(y) / (8y+11) \equiv -x^5 + x^4 + x^3 -x^2 +1 \pmod 3$ and $-x^5 + x^4 + x^3 -x^2 +1 $ is irreducible modulo $3$, therefore $h(y)$ is a product of distinct irreducible polynomials over $\mathbb{Z}$.
		We will now prove that there are infinitely many primes $p$ for which $h(p)$ is squarefree. To do this, we will use  Corollary~\ref{coro-2.2}. We claim that for all primes $r$, there exists $z \in (\mathbb{Z}/r^2\mathbb{Z})^*$ such that $h(z) \not\equiv 0 \pmod{r^2}$. Except for $r = 19$, we have
		$$
		h(1) = 5 \cdot 19^2 \cdot 1559 \not\equiv 0 \pmod{r^2}, 
		$$
		and $h(2) = 3^3 \cdot 2934361 \not \equiv 0 \pmod {19^2}.$ Therefore, there exist infinitely many primes $p$ such that $h(p)$ is squarefree. 
		\par 
		Note that, $g_A(x)$ has a real root $\rho$, with $-2< \rho< -1$. According to Corollary~\ref{Coro-2.10}, every zero $\rho \neq \pm 2$ of $g_A(x) $ corresponds to the two nonreal zeros,
		$$\omega = \frac{\rho + \sqrt{{\rho}^2 -4 }}{2} \quad \text{and}\quad \overline{\omega } = \frac{\rho - \sqrt{{\rho}^2 -4}}{2}$$
		of $f_A(x)$. Let 
		$$L= \mathbb{Q}(\omega)\quad \text{and}\quad K = \mathbb{Q}(\omega + \overline{\omega}) = \mathbb{Q}(\rho).$$
		We have $K \subset L$, with $[K:\mathbb{Q}]=5$ and $[L : K] = 2$.
		Suppose $p$ is a prime for which $h(p)$ is squarefree, then by \eqref{eq_4.2} and the formula $\Delta(g_p) = \Delta(K) \cdot [\mathbb{Z}_K : \mathbb{Z}[\rho]]^2$, we deduce that
		$$
		|\Delta(K)| =   8192p^5 + 125008p^4 + 156112p^3 - 139876p^2 - 15972p + 14641,
		$$
		and $g_p(x)$ is monogenic. Then, by Theorem~\ref{theorem-4.1}, we find that $\Delta(K)^2 | \Delta(L)$. Therefore, by \eqref{eq_4.1}, it follows that $|\Delta(f_p)|  = | \Delta(L)|$, so $f_p(x)$ is monogenic.
		\par 
		Next, we determine the Galois group of $g_p(x)$. Let $p$ be one of those primes for which the discriminant of $g_p(x)$ is squarefree. If $p \equiv -1 \pmod 3$, then $$g_p(x) \equiv x(x+1)(x^3 - x + 1) \pmod 3.$$ Therefore, Theorem~\ref{Thm-2.8} says that $\text{Gal}(g_p)$ contains a $3$-cycle. Using Theorem~\ref{Thm-2.9}, we find that $\text{Gal}(g_p)$ is $A_5$ or $S_5$. Since $\Delta(g_p)$ is squarefree, therefore $\text{Gal}(g_p)$ is $S_5$.
	\end{proof}
	Finally, we prove Theorem~\ref{th-2.14}.
	\begin{proof}[Proof of Theorem~\ref{th-2.14}]
	Using Proposition~\ref{lemma_1} for the polynomials $f(x)$ and $g(x)$, we have $f(x) = x^n g(x+x^{-1})$. Let $\alpha$ be a root of $f(x)$. Let $L=\mathbb{Q}(\alpha)$ and $K=\mathbb{Q}(\alpha + {\alpha}^{-1})$. Then, $[L:K]=2$ and $[K:\mathbb{Q}]=n$.	
		Let $\Delta(L)$ and $\Delta(K)$ denote the discriminants of the number fields $L$ and $K$, respectively. First, we prove that 
		\begin{equation}\label{eq-4.1}
		\begin{cases}
		\begin{array}{c@{\quad}l}
		\frac{\Delta(f)}{\Delta(L)} \mid f(1) f(-1) \left( \frac{\Delta(g)}{\Delta(K)} \right)^2  \quad \text{ i.e.,} \vspace{8pt} \\
		\operatorname{ind}(f)^2 \mid f(1)f(-1) \operatorname{ind}(g)^4.
		\end{array}
		\end{cases}
		\end{equation}
		Note that as per \eqref{eq-2.1}, we have
		\begin{align}
		&\Delta(f) = \operatorname{ind}(f)^2 \Delta(L)\ \text{ and }  \Delta(g) = \operatorname{ind}(g)^2 \Delta(K), \label{eq-2.17}  \\ 
		&\Delta(K)^2 \mid \Delta(L) \text{ (This follows from Proposition~\ref{theorem-4.1})}, \label{eq-2.18}  \\ 
		&\Delta(f) = (-1)^{n(2n-1)} f(1) f(-1) \Delta(g)^2. \label{eq-2.19}
		\end{align} 
		Now, \eqref{eq-2.18}  implies that $\Delta(L) = a\cdot \Delta(K)^2 $ for some $a \in \mathbb{Z}$. Therefore, using  \eqref{eq-2.17} and \eqref{eq-2.19} we obtain
		\begin{align*}
		&~\frac{f(1) f(-1) \Delta(g)^2 }{a \cdot \Delta(K)^2} = \frac{\Delta(f)}{\Delta(L)}, \\ 
		\text{i.e.}, &~\frac{a\cdot \Delta(f)}{\Delta(L)} = \frac{f(1)f(-1)\Delta(g)^2}{\Delta(K)^2},\\ 
		\text{i.e.}, &~\frac{\Delta(f)}{\Delta(L)}   \mid  f(1)f(-1)\left(\frac{\Delta(g)}{\Delta(K)}\right)^2.
		\end{align*} 
		Thus, we have $$ \operatorname{ind}(f)^2  \mid f(1) f(-1) \operatorname{ind}(g)^4.$$
		This completes the proof of \eqref{eq-4.1}. Since $g(x)$ is monogenic, we have $\operatorname{ind}(g) = 1$. Therefore, $\operatorname{ind}(f)^2 \mid f(-1)f(1)$. Also, $f(-1)f(1)$ is squarefree, so $\operatorname{ind}(f)^2 = 1$. Hence, $f(x)$ is monogenic.
	\end{proof}
\begin{example}
	Let $g(x) = x^4 + x^3 -3x^2 + 26x +57$ and $f(x) = x^4g(x+x^{-1})$. Using \texttt{SageMath} one can calculate that $\operatorname{ind}(g) = 1$. Here note that $f(1)f(-1) = 11^2$ is not squarefree and $11 \mid  \operatorname{ind}(f)$. This shows that condition $(1)$ in Theorem~\ref{th-2.14} is required. 
	\par We discuss another example where the calculations are done without using \texttt{SageMath}. Consider the reciprocal polynomial $f(x) = x^{10} + 7x^8 +16x^6 + 2x^5 + 16x^4 +7x^2 +1$. We obtain $g(x)= x^5 + 2x^3 + 2$ so that $f(x) = x^5g(x+x^{-1})$. Note that in view of \cite[Example 3.3]{Khanduja}, $g(x)$ is monogenic. Here, $5^2 \mid f(1) = 2 \cdot 5^2$. We will show that $5$ divides the index of $f(x)$. Note that, by division algorithm, we have $f(x) = (x-1)^2q(x) + 5^2(100x-80)$, where $q(x) \in \mathbb{Z}[x]$. Thus, $f(x) \in \langle 5, x-1\rangle^2$. Hence, by Theorem~\ref{Thm-2.6}, $5$ divides the index of $f(x)$. This confirms that condition $(1)$ in Theorem \ref{th-2.14} is required.
\end{example}
The next example is about condition $(2)$ of Theorem~\ref{th-2.14}.
\begin{example}
	Consider the reciprocal polynomial $$f(x) = x^{10}+ 26x^8 + 73x^6 +21x^5 + 73x^4 +26x^2+1.$$ We have $g(x) = x^5 + 21x^3 + 21$ such that $f(x) = x^5g(x+x^{-1})$.  In view of \cite[Example 3.3]{Khanduja}, $g(x)$ is not monogenic as the index of $g(x)$ is $29$. Also, $f(1) f(-1) = -13\cdot 17\cdot 179$ is a squarefree integer. By division algorithm we can write $f(x) = (x-2)^2 q(x) + 49880x-85463$, where $q(x) \in \mathbb{Z}[x]$. Thus, $f(x) = (x-2)^2 q(x) + 29 \cdot 1720(x-2) +17 \cdot 29^2$. Hence, $f(x) \in \langle 29, x -2 \rangle^2$ and $29$ divides the index of $f(x)$. This example shows that condition (2) in Theorem~\ref{th-2.14} is required.
\end{example}
In the following example, we discuss the converse of Theorem~\ref{th-2.14}.
\begin{example}
	Consider the reciprocal polynomial $f(x) = x^6 + px^3 + 1$. In \cite[Theorem~1]{jones_2021-2}, Jones established the existence of infinitely many primes $p$ for which $f(x)$ is monogenic. In particular, for $p \in S = \{3,5,13,17,19,31\}$, the polynomial $f(x)$ is monogenic. 
	\par 
Furthermore, we obtain
	$$
	g(x) = x^3 - 3x^2 + p
	$$
	satisfying the relation $f(x) = x^3 g(x + x^{-1})$. The discriminant of $g(x)$ is given by $3^2(p-2)(p+2)$. For an odd prime $p$, it follows that $2 \nmid (p-2)(p+2)$. Hence, by \cite[Proposition~4.36]{Narkiewicz}, the polynomial $g(x)$ is also monogenic. 
	Also, the product $f(-1)f(1)$ is squarefree for $p \in S$. This example suggests that the converse of Theorem~\ref{th-2.14} may also hold.
\end{example}
\begin{example}
	Let $a \neq 2 $ be an integer with $\gcd(a,3) = 1$. We consider the polynomial $g(x) = x^{3^r} + ax + 1$. We find that $\Delta(g)=-3^{r\cdot 3^r} - a^{3^k}(1-3^k)^{3^k-1}$. By \cite[Theorem~1.1(ii)]{Jhorar},  $g(x)$ is monogenic if and only if $\Delta(g)$ is squarefree. Consider the polynomial $f(x) = x^{3^k}g(x + x^{-1})$. By Theorem~\ref{th-2.14}, $f(x)$ is monogenic if $(2^{3^k} + 2a + 1)(1 -2a - 2^{3^k})\Delta(g)$ is squarefree.
\end{example}
To prove Corollary~\ref{Prop-1.7}, we recall a theorem of Kaur \textit{et al.} concerning the monogeneity of power compositional polynomials.
\begin{theorem} \cite[Theorem 1.1]{Kaur-2025} \label{prop-4.11}
	Let $f(x)$ be a monic polynomial with integer coefficients. Let $k \geq 2 $ be an integer such that $f(x^k) $ is irreducible. Then $f(x^k)$ is monogenic if and only if the following conditions hold.
	\begin{enumerate}
		\item $f(x)$ is monogenic,
		\item $f(0)$ is squarefree,
		\item $p$ does not divide the index of $f(x^k)$ for all primes $p \mid k$.
	\end{enumerate}
\end{theorem}
	\begin{proof}[Proof of Corollary \ref{Prop-1.7}]
		Since $f(x)$ is a monic reciprocal polynomial, we have $f(0) = 1$. Moreover, the product $f(-1)f(1)$ is squarefree and $g(x)$ is monogenic. Consequently, by Theorem~\ref{th-2.14}, it follows that $f(x)$ is monogenic. Hence, in view of Theorem~\ref{prop-4.11}, the proof follows.
	\end{proof}
\begin{example}
	Consider the reciprocal polynomial $f( x ) = x^4 + 3x^3 + 5x^2 +3x +1$. We obtain $g(x) = x^2 + 2x + 2$ such that $f(x) = x^2 g(x + x^{-1})$. Note that $g(x)$ is monogenic, $f(-1)f(1)$ is squarefree and $3$ does not divide the index of $f(x^3)$. Therefore, by Corollary~\ref{Prop-1.7}, $f(x^{3^k})$ is monogenic for every $k$. 
\end{example}
\section{Proof of Proposition \ref{prop-3.1} and \ref{Prop-4.1}}
We recall the following two results of Uchida \cite{uchida}, which will be used in the proofs of the propositions below. Although Uchida originally established these results for Dedekind rings, we state them here in the special case of the rings of integers of number fields.
\begin{theorem}\cite[Theorem]{uchida} \label{Thm-2.6}
	Let $K = \mathbb{Q}(\theta)$ be the number field generated by a root of the monic irreducible polynomial $f(x)$. Let $\mathbb{Z}_K$ be the ring of integers of $K$. Then $\mathbb{Z}_K = \mathbb{Z}[\theta]$ if and only if $f(x)$ does not belong to $\mathcal{M}^2$ for any maximal ideal $\mathcal{M}$ of the polynomial ring $\mathbb{Z}[x]$.
\end{theorem}
\begin{lemma}\cite[Lemma]{uchida}\label{lemma-uchida}
	An ideal $\mathcal{M}$ of $\mathbb{Z}[x]$ containing a monic polynomial is maximal if and only if $\mathcal{M} = \langle p, h(x) \rangle$ for some prime $p$ and a monic polynomial $h(x)$ belonging to $\mathbb{Z}[x]$ which is irreducible modulo $p$.
\end{lemma}
As an immediate corollary of Theorem~\ref{Thm-2.6} and Lemma~\ref{lemma-uchida}, we obtain the following result.
\begin{corollary}\label{cor-new-01}
Let $K = \mathbb{Q}(\theta)$ be the number field generated by a root of the monic irreducible polynomial $f(x)$. Then, a prime $p$ divides the index $[\mathbb{Z}_K : \mathbb{Z}[\theta]]$ if and only if $f(x)$ belongs to square of a maximal ideal of the form $\mathcal{M}= \langle p , h(x) \rangle$, where $h(x) \in \mathbb{Z}[x]$ is a monic polynomial which is irreducible modulo $p.$
\end{corollary}
\begin{proof}[Proof of Proposition \ref{prop-3.1}]
	Suppose $p \mid  \operatorname{ind}(g)$. Then by Corollary~\ref{cor-new-01}, $g(x) \in\langle p ,h(x) \rangle^2$ for some $h(x)  \in  \mathbb{Z}[x]$ which is irreducible modulo $p$. This yields$$g(x ) = p^2 a_1(x) + p h(x) a_2(x) + a_3(x) h(x)^2,$$ 
	where $a_i(x) \in \mathbb{Z}[x]$. Thus, we have
	$$x^n g(x+x^{-1}) = p^2 a_1(x+x^{-1}) x^n + p h(x+x^{-1}) a_2(x+x^{-1}) + a_3(x+ x^{-1}) h(x+x^{-1})^2 x^n .$$ Note that $f(x) = x^n g(x+x^{-1})$. Now,
	if $h_1(x) $ is an irreducible factor of the polynomial $h(x+x^{-1}) x^{\deg (h)}$, then $f(x) \in \langle  p, h_1(x) \rangle ^2$ so that $p$ divides the index of $f(x)$. 
\end{proof}
\begin{proof}[Proof of Proposition \ref{Prop-4.1}]
	First, we show that  $f(x)$ and $x^nf(1/x)$ have the same set of primes dividing their discriminants. Let $\widetilde{f}(x)= x^n f(\frac{1}{x}) $ and $\alpha_i$ be a root of $f(x)$. Then, $1/ \alpha_i$ is a root of $\widetilde{f}(x)$. Since $f(x)$ is monic and $a_0 = 1$, we have $\prod_{i=1}^{n} \alpha_i = \prod_{i=1}^{n} (1/\alpha_i) = 1$. Thus,
	\begin{align*}
	\Delta(\widetilde{f}) & = \prod_{i<j} (\frac{1}{\alpha_i} - \frac{1}{\alpha_j} )^2 
	= \prod_{i<j} \frac{(\alpha_i - \alpha_j)^2}{(\alpha_i \alpha_j)^2} \\
	& =\frac{\prod_{i<j} {(\alpha_i - \alpha_j)^2}}{(\prod_{i=1}^{n} \alpha_i)^{2n-2}} 
	= \prod_{i<j} {(\alpha_i - \alpha_j)^2} \\
	& = \Delta(f). 
	\end{align*}
	Suppose a prime $p$ divides the index of $f(x)$.  Then, again by Corollary~\ref{cor-new-01}, $f(x) \in \langle p , h(x) \rangle^2$, where $h(x)$ is an irreducible polynomial modulo $p$. Hence,
	$$f(x) = pa_0(x)+ph(x)a_1(x)+h(x)^2a_2(x)$$
	for some $a_i(x) \in  \mathbb{Z}[x]$. This yields
	$$x^nf(1/x)  = pa_0(1/x)x^n+ph(1/x)a_1(1/x)x^n+h(1/x)^2a_2(1/x)x^n.$$
	Let $h_1(x)$ be the irreducible factor of $h(1/x)x^{\operatorname{deg}(h(x))}$ modulo $p$. Then, we have $x^nf(1/x)  \in \langle p , h_1(x) \rangle^2$. Therefore, $p$ divides the index of $x^nf(1/x)$. The converse can be proved in a similar manner. 
\end{proof}
\section{Proof of Theorem  \ref{Prop-1.8}}
		In this section, we prove Theorem~\ref{Prop-1.8}. We begin by establishing a lemma that provides a sufficient condition for the monogeneity of a certain sextic polynomial. This lemma will play a key role in the proof of Theorem~\ref{Prop-1.8}.
	\begin{lemma} \label{Lemmma-1}
		Consider the polynomial $H(x) = (4x^2 + 20 x - 7 )(8x + 7)$. If there exists an integer $a$ for which $H(a)$ is squarefree, then the polynomial $$f_a(x) = x^6 + x^5 + (2a + 1)x^4 + (4a + 1)x^3 + (2a+1)x^2 + x + 1$$ is monogenic. 
	\end{lemma}
	\begin{proof}
		Consider the following integral reciprocal polynomial
		$$f_T(x)  = x^6 + x^5 + (2T + 1)x^4 + (4T + 1)x^3 + (2T+1)x^2 + x + 1.$$ 
		We obtain $g_T(x)  = x^3 + x^2 + (2T-2)x + 4T -1$. It is straightforward to verify that both $f_T(x)$ and $g_T(x)$ are irreducible modulo $2$, and therefore irreducible over $\mathbb{Q}$. Using \texttt{SageMath}, we calculate the discriminants of $f_T(x)$ and $g_T(x)$:
		\begin{align*} 
			\Delta(f_T) & =  -(4T^2+20T -7)^2(8T + 7)^3, \\
			\Delta(g_T) & =  (4T^2+20T -7)(8T + 7).
		\end{align*}
		If $\alpha$ is a root of $f_T(x)$ then $ \alpha + \alpha^{-1}$ is a root of $g_T(x)$. Therefore, we have
		$$ \mathbb{Q} \subset K = \mathbb{Q}(\alpha + \alpha^{-1}) \subset L = \mathbb{Q}(\alpha) \text{ with }  [L: K] = 2, ~ [K:\mathbb{Q}] = 3.$$ Let $a \in \mathbb{Z} $ be such that $H(a)$ is squarefree. Then using \eqref{eq-2.1}, we obtain $\Delta(K)  = \Delta(g_a)$. Moreover, $\Delta(K)^2 \mid \Delta(L)$, therefore, again using \eqref{eq-2.1}, we have $\Delta(L)  = \Delta(f_a)$. Hence, the index of $f_a(x)$ is one, so that $f_a(x)$ is monogenic.
	\end{proof}
	Next, we recall two standard notations commonly used in analytic number theory. First, for arithmetic functions $f(n)$ and $g(n)$, we write $g(n) \gg f(n)$  
	if there exists a constant $K > 0$ such that $|f(n)| < K \cdot g(n)$ for all sufficiently large $n$. 
	Furthermore, if 
	\begin{equation*}
	\lim_{x \to \infty} \frac{f(x)}{g(x)} = 1,
	\end{equation*}
	then we say that $f(x)$ is asymptotic to $g(x)$ as $x \to \infty$, and we denote this by 
	\begin{equation*}
	f(x) \sim g(x) \quad (x \to \infty).
	\end{equation*}
	We now define the following notations, which we will use in the proof of Theorem~\ref{Prop-1.8}. \begin{align*}
		\mathcal{N}_H(N) & = \# \{ a\leq N \mid a ~ \text{is a prime and }~ H(a) ~ \text{is~squarefree} \}, \\
		\mathcal{M}_H(N)  & = \# \{a \leq N \mid H(a) ~ \text{is squarefree}\}, \\
		\mathcal{L}_f(N)  & = \# \{a \leq N \mid f_a(x) ~ \text{is monogenic}\}.
	\end{align*}	
	\begin{proof}[Proof of Theorem \ref{Prop-1.8}]
		Using Lemma~\ref{Lemmma-1}, for sufficiently large $N$, we have $$\mathcal{L}_f(N) \gg \mathcal{M}_H(N) \gg \mathcal{N}_H(N).$$ Now we will show that for sufficiently large $N$, $\mathcal{N}_H(N) \gg \frac{N}{\log N}.$ Note that $H(1)  = 17 \cdot 15 \not\equiv 0 \pmod {q^2}$ for any prime $q$. This shows that $H(x)$ has no local obstruction at any prime $q$. Therefore, by Theorem~\ref{thm-2.2}, we have $\mathcal{N}_H(N) \sim c_f  ~ \frac{N}{\log N}$. Thus, for sufficiently large $N$, we obtain $\mathcal{N}_H(N) \gg \frac{N}{\log N}$. This completes the proof of the theorem.
	\end{proof}
\begin{remark}
	We remark that a similar lower bound can be obtained for the following classes of polynomials studied in \cite{jones_2021-2, jones_2021-1}: 
	\begin{enumerate}
		\item $g_a(x) = x^6 - 7x^4 + 14x^2 + 196a -7$,
		\item $h_a(x) = x^6 + 3x^5 + (a +6)x^4 + (2a + 7)x^3 + (a+6)x^2 + 3x +1$.
	\end{enumerate}
	\end{remark}
	\section{Concluding remarks}
		Proposition~\ref{prop-3.1} gives a partial step toward establishing the converse of Theorem~\ref{th-2.14}. It is an interesting problem to determine whether the converse is indeed true. Some progress has been made in the relative setting, that is, when the base field is not $\mathbb{Q}$. It would therefore be worthwhile to investigate whether Theorem~\ref{th-2.14} can be extended to the relative case.
		
	\section{Data Availability Statements}
	Data sharing not applicable to this article as no datasets were generated or analysed during the current study.
	\section{Conflict of interest}
	The authors assert that there are no conflicts of interest.


\begin{thebibliography}{10} 

	\bibitem{Alexandersson}
	P.~Alexandersson, L.~A. Gonz\'alez-Serrano, E.~A. Maximenko, and M.~A.
			Moctezuma-Salazar, \textit{Symmetric polynomials in the symplectic alphabet and the change of
	variables $z_j = x_j + x^{-1}_j$}, Electron. J. Combin. 28 (2021), no. 1, Paper No. 1.56, 36 pp.
	
	\bibitem{cohen}
	H.~Cohen, {\em A course in computational algebraic number theory}, Graduate Texts in Mathematics, volume 138, Springer-Verlag, Berlin, 1993.
	
	\bibitem{conrad-2}
	K.~Conrad, https://kconrad.math.uconn.edu/blurbs/galoistheory/galoisSnAn.pdf
	
	\bibitem{conrad}
	K.~Conrad, https://kconrad.math.uconn.edu/blurbs/gradnumthy/galois-Q-factor-mod-p.pdf
	
	\bibitem{Elouafi-2014}
	M.~Elouafi, {\em On a relationship between Chebyshev polynomials and Toeplitz
	determinants}, Applied Mathematics and Computation 229 (2014), pp. 27--33.
	
	\bibitem{hasse}
	H.~Hasse, {\em Zahlentheorie}, Akademie-Verlag, Berlin, 1963.
	
	\bibitem{gaal-2}
	I.~Ga\'al, {\em Monogenity and power integral bases: recent developments}, Axioms 13 (2024), no. 7, pp. 429--429.
	
	\bibitem{Ravi}
	A.~Jakhar, R.~Kalwaniya, and S.~Kotyada, {\em On monogenity of number fields and {Galois} group}, 
	Int. J. Number Theory 21 (2025), no. 8, pp. 1995--2013.
	
	\bibitem{Khanduja}
	A.~Jakhar, S.~K. Khanduja, and N.~Sangwan, {\em Characterization of primes dividing the index of a trinomial}, Int. J. Number Theory 13 (2017), no. 10, pp. 2505--2514.
	
	\bibitem{Anurabha-2025b}
	A.~Jakhar, S.~Kotyada, and A.~Mukhopadhyay, {\em On squarefree parts of polynomial discriminants and quadratic fields}, Results Math. 80 (2025), no. 5, Paper No. 132, 9 pp.
	
	\bibitem{prabhakar}
	A.~Jakhar, S.~Laishram, and P.~Yadav, {\em Explicit discriminant of a class of polynomial, monogenity and
	{Galois} group}, Commun. Algebra 53 (2025), no. 7, pp. 2937--2948.
	
	\bibitem{Anurabha-2025a}
	A.~Jakhar, K.~Srinivas, and A.~Mukhopadhyay, {\em Monogenic polynomials and symmetric {G}alois groups: a quantitative study}, Ramanujan J. 67 (2025), no. 3, Paper No. 69, 9 pp.
	
	\bibitem{Jhorar}
	B.~Jhorar and S.~K. Khanduja, {\em On power basis of a class of algebraic number fields}, 
	Int. J. Number Theory 12 (2016), no. 8, pp. 2317--2321.
	
	\bibitem{jones_2021-2}
	L.~Jones, {\em Sextic reciprocal monogenic dihedral polynomials}, 
	Ramanujan J. 56 (2021), no. 3, pp. 1099--1110.
	
	\bibitem{jones_2021-1}
	L.~Jones, {\em Infinite families of reciprocal monogenic polynomials and their
	{G}alois groups}, New York J. Math. 27 (2021), pp. 1465--1493.
	
	\bibitem{Jones-2022}
	L.~Jones, {\em Reciprocal monogenic quintinomials of degree {$2^n$}}, Bull. Aust. Math. Soc. 106 (2022), no. 3, pp. 437--447.
	
	\bibitem{Jones-2025}
	L.~Jones, {\em Reciprocal monogenic septinomials of degree {$2^n3$}}, Ann. Math. Sil. 39 (2025), no. 1, pp. 155--169.
	
	\bibitem{Kaur-2025}
	S.~Kaur, S.~Kumar, and L.~Remete, {\em On the index of power compositional polynomials}, 
	Finite Fields Appl. 107 (2025), Paper No. 102642, 20 pp.
	
	\bibitem{Narkiewicz}
	W.~Narkiewicz, {\em Elementary and analytic theory of algebraic numbers}, Springer-Verlag, Berlin; PWN---Polish Scientific Publishers, Warsaw,
	second edition, 1990.
	
	\bibitem{pasten}
	H.~Pasten, {\em The {ABC} conjecture, arithmetic progressions of primes and
	squarefree values of polynomials at prime arguments}, Int. J. Number Theory 11 (2015), no. 3, pp. 721--737.
	
	\bibitem{uchida}
	K.~Uchida, {\em When is {$Z[\alpha ]$} the ring of the integers?}, Osaka Math. J. 14 (1977), no. 1, pp. 155--157.
\end{thebibliography}
\end{document}